\documentclass[12pt,reqno]{amsart}
\let\oldlabel=\label
\def\prellabel{\marginparsep=1em\marginparwidth=44pt
  \def\label##1{\oldlabel{##1}\ifmmode\else\ifinner\else
         \marginpar{{\footnotesize\ \\ \tt
                    ##1}}\fi\fi}}
\def\NZQ{\mathbb}               
\def\NN{{\NZQ N }}

\def\ZZ{{\NZQ Z}}

\def\opn#1#2{\def#1{\operatorname{#2}}} 
\opn\chara{char}
\opn\rank{rank}
\opn\hilb{Hilb}
\opn\gr{gr}
\opn\Rees{{\mathcal R}}
\let\dirsum=\oplus

\let\Dirsum=\bigoplus

\usepackage[latin1]{inputenc}
\usepackage{mathptmx}
\newtheorem{theorem}{Theorem}[section]
\newtheorem{lemma}[theorem]{Lemma}

\newtheorem{proposition}[theorem]{Proposition}

\theoremstyle{definition}
\newtheorem{remark}[theorem]{Remark}
\newtheorem{definition}[theorem]{Definition}
\newtheorem{example}[theorem]{Example}
\newtheorem{remark/example}[theorem]{Remark/Example}

\newtheorem*{theorem*}{Theorem}

\let\epsilon=\varepsilon
\let\phi=\varphi
\let\kappa=\varkappa

\opn\ini{in}
\opn\KRS{KRS}
\opn\krs{krs}
\opn\Krs{Krs}
\opn\DEL{DEL}
\opn\diag{diag}
\opn\Ker{Ker}
\opn\Image{Im}
\opn\DD{{\mathcal D}}
\opn\SS{{\mathcal S}}
\opn\MM{{\mathcal M}}
\opn\GL{GL}
\opn{\hht}{ht}
\opn\Cl{Cl}
\opn\cl{cl}
\opn\height{height}

\def\cS{{\mathcal S}}
\def\cI{{\mathcal I}}
\def\cJ{{\mathcal J}}
\def\cB{{\mathcal B}}

\def\mm{{\mathfrak m}}

\let\ra=\rangle
\let\la=\langle

\def\lex{{\textup{lex}}}
\def\str{{\textup{str}}}
\def\revlex{{\textup{revlex}}}

\def\addots{\mathinner{\mkern1mu\raise1pt\hbox{.}\mkern2mu\raise4pt\hbox{.}
        \mkern2mu\raise7pt\vbox{\kern7pt\hbox{.}}\mkern1mu}}

\def\discuss#1{\marginpar{\marginparwidth=2em\raggedright\tiny #1}}

\numberwithin{equation}{section}

\author{Winfried Bruns}
\address{Universit\"at Osnabr\"uck, Institut f\"ur Mathematik, 49069 Osnabr\"uck, Germany}
\email{wbruns@uos.de}

\author{Aldo Conca}
\address{ Dipartimento di Matematica,
Universit\`a degli Studi di Genova, Italy}
\email{conca@dima.unige.it}

\title{Products of Borel fixed ideals of maximal minors}


\date{}

\begin{document}

\maketitle

\begin{abstract}
We study a large family of products of Borel fixed ideals of maximal minors. We compute their initial ideals and primary decompositions, and show that they have linear free resolutions. The main tools are an extension of straightening law and a very surprising primary decomposition formula. We   study  also the homological properties of  associated multi-Rees algebra which are shown to be  Cohen-Macaulay, Koszul and defined by a Gr\"obner basis of quadrics.
\end{abstract}


\section{Introduction}
Let $K$ be a field and $X=(x_{ij})$ be the $m\times n$ matrix whose entries are the indeterminates of the polynomial ring $R=K[x_{ij}: 1\le i\le m,\ 1\le j \le n]$, and assume that $m\leq n$. The ideals $I_t(X)$, generated by the $t$-minors of $X$, and their varieties are classical objects of commutative algebra, representation theory  and algebraic geometry. They are clearly invariant under the natural action of $\GL_m(K)\times \GL_n(K)$ on $R$. Their arithmetical and homological properties are well-understood as well as their Gr\"obner bases and initial ideals with respect to diagonal (or antidiagonal) monomial  orders, i.e., monomial orders under which the initial monomial of a minor is the product over its diagonal (or antidiagonal); see our survey \cite{BC2}.  Bruns and Vetter \cite{BV} and Miller and Sturmfels \cite{MS} are comprehensive treatments.  

Among the ideals of minors the best-behaved   is  undoubtedly  the ideal of maximal minors, namely the ideal $I_m(X)$. It has  the following important features: 

\begin{theorem}
\label{wellknown}\leavevmode
\begin{enumerate} 
\item The   powers of $I_m(X)$ have a linear resolution. 
\item  They are primary and integrally closed.
\item  Computing initial ideals  commutes with taking powers for diagonal or anti-diagonal monomial orders: $\ini(I_m(X)^k)=\ini(I_m(X))^k$ for all $k$, and the natural generators of $I_m(X)^k$ form a Gr\"obner basis. 
\item The Rees algebra of $I_m(X)$ is Koszul, Cohen-Macaulay and normal. 
\end{enumerate} 
\end{theorem} 

In the theorem and throughout this article ``resolution'' stands for ``minimal graded free resolution''. The grading is always the standard grading on the polynomial ring.

Concerning the statements in (1), one knows that  $I_m(X)$ itself is resolved by the Eagon-Northcott complex and the resolution for the powers is described  by  Akin, Buchsbaum and Weyman  in \cite{ABW}. References for  assertions (2),  (3) and (4) can be found in \cite{BC1}, \cite{BC2}, \cite{BCV}, \cite{BV}, and Eisenbud and Huneke \cite{EH}.  Note also that the  maximal minors form a universal Gr\"obner basis (i.e., a Gr\"obner basis with respect to every monomial order) as proved by Bernstein, Sturmfels and Zelevinsky in \cite{BZ}, \cite{SZ} and generalized by Conca, De Negri, Gorla \cite{CDG}. But for  $m>2$ and $k>1$ there are  monomial orders $<$ such that  $\ini_<(I_m(X)^k)$ is strictly larger than  $\ini(I_m(X))^k$. In other words, the  natural generators of $I_m(X)^k$ do not form a universal Gr\"obner basis. This is  related to the fact that the maximal minors do not form a universal Sagbi basis for the coordinate ring of the Grassmannian, as observed, for example, by Speyer and Sturmfels \cite[5.6]{SS}. 
 
 For $1<t<m$ the ideal  $I_t(X)$  does not have a linear resolution  and its powers are not primary. The primary decomposition of the powers of $I_t(X)$ has been computed by De Concini, Eisenbud and Procesi \cite{DEP}  and in \cite{BV} . The Castelnuovo-Mumford  regularity of $I_t(X)$ is computed  by Bruns and Herzog \cite{BH}. Furthermore, the formation of initial ideals does not  commute with taking powers,   but $I_t(X)^k$  has a  Gr\"obner basis in degree $tk$ as the results in \cite{BC1} show. 

In our joint work with Berget \cite{BBC},  Theorem \ref{wellknown}   was extended to arbitrary products  of the ideals $I_t(X_t)$ where $X_t$ is the submatrix of the first $t$ rows of $X$.   We proved the following results: 

\begin{theorem}
\label{berget}
Let   $1\leq t_1,\dots, t_w\leq m$ and   $I=I_{t_1}(X_{t_1})\cdots I_{t_w}(X_{t_w})$. 
\begin{enumerate} 
\item Then $I$ has a linear resolution. 
\item  $I$ is integrally closed and it has  a primary decomposition whose components are powers of ideals $I_t(X_t)$ for various values of $t$. 
\item   $\ini(I)=\ini(I_{t_1}(X_{t_1})) \cdots \ini(I_{t_w}(X_{t_w}))$ and the natural  generators of $I$ form a Gr\"obner basis with respect to a diagonal or anti-diagonal monomial order. 
\item The multi-Rees algebra    associated to $I_{1}(X_{1}), \dots,  I_{m}(X_{m})$ is Koszul, Cohen-Macaulay and normal. 
\end{enumerate} 
\end{theorem} 

Note that the ideals $I_t(X_t)$ are fixed by the natural action of the subgroup $B_m(K)\times \GL_n(K)$ of $\GL_m(K)\times \GL_n(K)$, where $B_m(K)$ denotes the subgroup of lower  triangular matrices. For use below we denote the subgroup of upper triangular matrices in $\GL_n(K)$ by $B'_n(K)$ .

Ideals of minors that are invariant under the Borel group $B_m(K)\times B'_n(K)$ have been  introduced and studied by Fulton in  \cite{F}. They  come up  in the study of singularities of various kinds of Schubert subvarieties of the Grassmannians and flag varieties. 
Those that arise as  Borel orbit closures of (partial) permutation matrices are called Schubert determinantal ideals by  Knutson and Miller in the their beautiful paper \cite{KM}  where they describe the associated Gr\"obner bases, as well as Schubert and Grothendieck polynomials. 

The goal of this paper is to extend the results of  Theorems \ref{wellknown} and \ref{berget} to a class of ideals that are fixed by the Borel group. Depending on whether one takes upper or lower triangular matrices on the left or on the right, one ends up with different ``orientations", in the sense that for $B_m(K)\times B'_n(K)$ one gets ideals of minors that flock the northwest corner of the matrix while for $B_m(K)\times B_n(K)$ the ideals of minors flock the northeast corner, and so on.  Of course, there is no real difference between the four cases, but because we prefer to work with diagonal monomial orders,  we will choose  the northeast orientation. Clearly, all the results we prove can be formulated in terms of the other three orientations as well. 

Let us define the \emph{northeast ideals $I_t(a)$ of maximal minors}: $I_t(a)$ is generated by the $t$-minors of the $t\times (n-a+1)$ \emph{northeast submatrix} 
$$
X_t(a)=(x_{ij}: 1\le i \le t,\ a\le j \le n).
$$
The main results can be summed up as follows:

\begin{theorem} 
\label{main} 
Let $I_{t_1}(a_1),\dots, I_{t_w}(a_w)$ be northeast ideals  of maximal minors, and let $I$ be  their product. Then
\begin{enumerate}
\item $I$ has a linear resolution.
\item $\ini(I)=\ini(I_{t_1}(a_1)) \cdots    \ini(I_{t_w}(a_w))$,  and the natural generators of $I$ form a Gr\"obner basis with respect to a diagonal monomial order. 
\item $I$ is integrally closed, and it has a  primary decomposition  whose components are powers of ideals $I_t(a)$ for various values of $t$ and $a$.   
\item The multi-Rees algebra    associated to  the family of ideals $I_{t}(a)$ with $t,a>0$ and $t+a\leq n+1$ is Koszul, Cohen-Macaulay and normal. 
\end{enumerate}
\end{theorem} 

Statements  (1), (3) and (4) hold analogously for the initial ideals, in particular the primary components of $\ini(I)$  can be taken to be powers of ideals of variables. 

One could consider a more general definition of northeast ideals  of maximal minors, allowing also more rows than columns. 
Unfortunately the results of Theorem \ref{main} do not hold in this generality.  Let $I'_t(a)$ denote the ideal of the $t$-minors in the submatrix  formed by the last $t$ columns and first $a$ rows (with $a>t$). For example,  one can check in a $3\times 3$ matrix   that the product of (Borel-fixed ideals of maximal minors) $I_1(2)I_2(1)I'_1(2)I'_2(3)$ does not have a linear resolution.

The proofs of the results of \cite{BBC} are  based on the straightening law since the ideals considered in \cite{BBC} have $K$-bases of standard bitableaux. This is no longer true for the ideals $I_t(a)$ in general, let alone for products of such ideals. Therefore we had to develop a more general notion of ``normal form'' that we call \emph{northeast canonical}. Using this type of normal form we will prove the crucial description of the initial ideal $\ini(I)$ as an intersection of powers of the ideals $\ini(I_t(a))$.

The northeast canonical form allows us to prove that the multi-Rees algebra defined by all ideals $I_t(a)$ is a normal domain and is defined by   a Gr\"obner basis of quadrics. A theorem of Blum \cite{B} then implies that all our ideals have linear fee resolutions. The same statements have counterparts for the initial ideals and their multi-Rees algebra as well.

We conclude the paper with a discussion of the Gorenstein property of certain  multigraded Rees rings and  the factoriality of certain fiber rings that come up in connection with the northeast ideals. In particular, we prove that the multigraded Rees algebra associated to a strictly ascending chain   
of ideals $J_1\subset J_2 \dots \subset J_v$ is Gorenstein, provided  each $J_i$ belongs to the family of the $I_t(a)$ and has height $i$.

The results of this paper originated from extensive computations with the systems CoCoA \cite{Cocoa}, Macaulay 2 \cite{M2}, Normaliz \cite{Nmz} and Singular \cite{DGPS}.
 
\section{Minors, diagonals and the straightening law}

Let $K$ be a field and $X=(x_{ij})$ an $m\times n$ matrix of indeterminates. The ideals we want to investigate live in  
$$
R=K[x_{ij}: 1\le i\le m,\ 1\le j\le n].
$$ 
Let $X_t(a)$ be the submatrix of $X$ that consists of the entries $x_{ij}$ with $1\le i\le t$ and $a\le j\le n$. We call it a \emph{northeast submatrix} since it sits in the right upper corner of $X$. The ideal
$$
I_t(a)=I_t(X_t(a))
$$
is called a \emph{northeast ideal of maximal minors}, or a \emph{northeast ideal} for short. In the following ``northeast'' will be abbreviated by ``NE''. Since $I_t(a)=0$ if $t+a>n+1$, we will always assume that $t+a\le n+1$.

We fix a monomial order on $R$ that fits the NE ideals very well: the lexicographic order $>_\lex$ (or simply $>$) in which $x_{11}$ is the largest indeterminate, followed by the elements in the first row of $X$, then the elements in the second row from left to right, etc. More formally:
$$
x_{ij}>_\lex x_{uv} \quad\text{if}\quad i<u \quad\text{or}\quad i=u \text{ and } j< v.
$$

The minor
$$
\delta=\det(x_{ib_j}:i,j=1,\dots,t),\qquad b_1< \dots<  b_t,
$$
is denoted by $[b_1 \dots b_t]$. The \emph{shape} $|\delta|$ is the number $t$ of rows.  The initial monomial of $\delta$  is the diagonal
$$
\la b_1 \dots b_t\ra =x_{1b_1}\cdots x_{tb_t}.
$$
Therefore $<$ is a \emph{diagonal} monomial order. All our theorems remain valid for an arbitrary diagonal monomial order $\prec$ since we will see that for our ideals $I$ the initial ideals $\ini_{<_\lex}(I)$ are generated by initial monomials of products of minors. Therefore $\ini_{<_\lex}(I)\subset \ini_\prec(I)$, and the inclusion implies equality. In view of this observation we will suppress the reference to the monomial order in denoting initial \emph{ideals}, always assuming that the monomial order is diagonal. However, when we compare single \emph{monomials}, the lexicographic order introduced above will be used. 

In the straightening law, Theorem \ref{straight}, we need a partial order for the minors and also for their initial monomials:
\begin{align*}
[b_1 \dots b_t]\le_\str [c_1 \dots c_u] &\iff  \la b_1 \dots b_t\ra\le_\str \la c_1 \dots c_u\ra\\
&\iff t\ge u\text{ and } b_i\le c_i,\ i=1,\dots,u.
\end{align*}
It is easy to see that the minors as well as their initial monomials form a lattice with the meet and join operations defined as follows: if $t\ge u$,
\begin{align*}	
[b_1 \dots b_t]\vee [c_1 \dots c_u]&= [c_1 \dots c_u]\vee [b_1 \dots b_t]=[\max(b_1,c_1),\dots,\max(b_u,c_u)],\\
[b_1 \dots b_t]\wedge [c_1 \dots c_u]&= [c_1 \dots c_u]\wedge [b_1 \dots b_t] = [\min(b_1,c_1),\dots,\min(b_u,c_u),b_{t+1},\dots,b_t].
\end{align*}
The meet and join of two diagonals are defined in the same way: just replace $[\cdots]$ by $\la\cdots\ra$.

A product
$$
\Delta=\delta_1\cdots \delta_p,\quad \delta_i=[b_{i1}\dots b_{it_i}],\ i=1,\dots,p,
$$
of minors is called a \emph{tableau}. The \emph{shape} of $\Delta$ is the $p$-tuple $|\Delta|=(t_1,\dots,t_p)$, provided $t_1\ge\dots \ge  t_p$, a condition that does not restrict us in any way.

If
$$ 
\delta_1\le_\str\dots \le_\str \delta_p
$$
then we say that $\Delta$ is a \emph{standard tableau}. In the context of determinantal ideals one usually has to deal with bitableaux, but in this paper the row indices are always fixed so that we only need to take care of the column indices. Since the product does not determine the order of its factors, one should distinguish the sequence of minors from the product if one wants to be formally correct; as usually, we tacitly assume that such products come with an order of their factors.

\begin{proposition}
\label{uni-init}
Let $\Delta$ be a tableau. Then there exists a unique standard tableau $\Sigma$ such that $\ini(\Delta)=\ini(\Sigma)$. Furthermore $\Delta$  and 
$\Sigma$ have the same shape. 
\end{proposition}

This is easy to see: if $\Delta=\delta_1\cdots \delta_p$ is not standard, then there exist $i$ and $j$ such that $\delta_i$ and $\delta_j$ are not comparable. Since $\ini(\delta_i\delta_j)=\ini((\delta_i\wedge\delta_j)(\delta_i\vee\delta_j))$ we can replace $\delta_i\delta_j$ by an ordered pair of factors, and after finitely many such operations we reach a standard tableau. It is evidently unique.

That the row indices are not only fixed, but always given by $1,\dots,t$ for a $t$-minor simplifies the \emph{straightening law}. 

\begin{theorem}\label{straight}\leavevmode
\begin{enumerate}
\item Let $\delta=[b_1\dots b_t]$ and $\sigma=[c_1\dots c_u]$ be minors. Then there exist uniquely determined minors $\eta_i,\zeta_i$ and coefficients $\lambda_i\in K$, $i=1,\dots,q$, $q\ge 0$, such that
\begin{multline*}
\delta\sigma=(\delta\wedge \sigma)(\delta\vee\sigma)+\lambda_1\eta_1\zeta_1+\dots+\lambda_q\eta_q\zeta_q,\\
\eta_i\le_\str \delta \wedge\sigma, \ \delta_i\vee\sigma \le_\str \zeta_i, \ i=1,\dots,q,\\
\ini(\delta\sigma)=\ini((\delta\wedge \sigma)(\delta \vee\sigma))>\ini(\eta_1\zeta_1)>\cdots > \ini(\eta_q\zeta_q).
\end{multline*}
 
\item For every tableau $\Delta$ there exist standard tableaux $\Sigma_0\dots,\Sigma_q$ of the same shape as $\Delta$ and uniquely determined coefficients $\lambda_1,\dots,\lambda_q$, $q\ge0$, such that
$$
\Delta=\Delta_0+\lambda_1\Sigma_1+\dots+\lambda_p\Sigma_p,\quad \ini(\Delta)=\ini(\Sigma_0)>\ini(\Sigma_1)>\dots>\ini(\Sigma_q).
$$
 
\end{enumerate}
\end{theorem}
Note that the $K$-algebra generated by the $t$-minors of the first $t$-rows for $t=1,\dots, m$ is   the coordinate of the flag variety. Hence Theorem \ref{straight} can be deduced from \cite[14.11]{MS}, and can also be derived from   \cite[(11.3) and (11.4)]{BV}, taking into account Proposition \ref{uni-init}. 

\section{Initial ideals and primary decomposition}

The main objects of this paper are products of ideals $I_t(a)$. We will access them  via  the initial ideals
$$
J_t(a)=\ini(I_t(a)).
$$
Our first goal is to determine the primary decompositions of such products along with their initial ideals. For the powers of a single ideal $I_t(a)$ the answer is well-known:

\begin{theorem}\label{classic}\leavevmode
\begin{enumerate}
\item The powers of the prime ideal $I_t(a)$ are primary. In other words, the ordinary and the symbolic powers of $I_t(a)$ coincide.
\item $J_t(a)$ is generated by the initial monomials $\ini(\delta)$ of the $t$-minors  of $I_t(a)$.
\item $\ini(I_t(a)^k)=J_t(a)^k$ for all $k\ge 1$.
\end{enumerate}
\end{theorem}

See \cite[(9.18)]{BV} for the first statement and \cite{Co4} for the remaining statements. The results just quoted are formulated for $a=1$, but they immediately extend to general $a$ since polynomial extensions of the ground ring are harmless.

The primary decompositions of the powers of  $J_t(a)$ have been determined in \cite[Prop. 7.2]{BC3}. We specify the technical details only as far as they are needed in this article:

\begin{theorem}\label{in-powwer-primary}
The ideal $J_t(a)$ is radical. It is the intersection $J_t(a)=\bigcap_i P_i$ of prime ideals $P_i$ that are generated by $(n-a-t+2)$ indeterminates,  and  $J_t(a)^k=\bigcap_i P_i^k$ for all $k$. In particular, $J_t(a)^k$ has no embedded primes and it is integrally closed.
\end{theorem}

For the precise description of the set of prime ideals $P_i$ appearing in Theorem \ref{in-powwer-primary} we refer the reader to \cite{BC3}.

Now we introduce the main players formally:

\begin{definition}
\label{patterns}
A \emph{NE-pattern} is a finite sequence $\bigl((t_1,a_1),\dots,(t_w,a_w)\bigr)$ of pairs of positive natural numbers with $t_i+a_i\le n+1$ for $i=1,\dots,w$ and  which is ordered according to the following rule: if  $1\leq i<j\leq w$, then
$$
a_i\le a_j\quad \text{and}\quad t_i\ge t_j \text{ if } a_i=a_j.
$$

Let $S=\bigl((t_1,a_1),\dots,(t_w,a_w)\bigr)$ be a NE-pattern. A \emph{pure NE-tableau of pattern} $S$ is a product of minors
$$
\Delta=\delta_1\cdots\delta_w,\quad \text{ such that } \delta_i \text{ is a $t_i$-minor of }   X_{t_i}(a_i) ,\ i=1,\dots,w.
$$

An \emph{NE-tableau} is a product $M\Delta$ of a monomial $M$ in the indeterminates $x_{ij}$ and a pure NE-tableau $\Delta$.

 The NE-\emph{ideal} of pattern $S$ is the ideal generated by all (pure) NE-tableaux of pattern $S$. In other words, it is the ideal
$$
I_S=I_{t_1}(a_1)\cdots I_{t_w}(a_w).
$$
Furthermore we set
$$
J_S=\ini(I_S).
$$
\end{definition} 

So $I_S$ is  simply a product of  ideals of type $I_{t}(a)$ where, by convention,  the factors have been ordered according to the rule specified in \ref{patterns}. 

For $S=\bigl((t_1,a_1),\dots,(t_w,a_w)\bigr)$ and a pair $(u,b)$ we set
$$
e_{ub}(S)=|\{i: b \le a_i \text{ and } u\le t_i \}|.
$$
Note that $b \le a_i \text{ and } u\le t_i$ is indeed equivalent to  $I_{t_i}(a_i)\subset I_u(b)$.

\begin{theorem}
\label{primary}
Let  $S=\bigl((t_1,a_1),\dots,(t_w,a_w)\bigr)$ be a NE-pattern. Then  the following hold:
\begin{align}
J_S&=J_{t_1}(a_1)\cdots J_{t_w}(a_w);\label{ini}\\
J_S&=\bigcap_{u,b}J_u(b)^{e_{ub}(S)};\label{ini-inter}\\
I_S&=\bigcap_{u,b}I_u(b)^{e_{ub}(S)}.\label{inter}
\end{align}
Equation \eqref{inter} gives a primary decomposition of $I_S$. The ideals $I_S$ and $J_S$ are integrally closed.
\end{theorem}

 As soon as the equation \eqref{inter} will have been proved, it indeed yields a primary decomposition of $I_S$ since all the ideals $I_u(b)^e$ are primary, being powers of ideals of maximal minors. The intersection in \eqref{inter} is almost always redundant. An irredundant decomposition will be described in  Proposition \ref{irred}. Together with Theorem \ref{in-powwer-primary}, equation \eqref{ini-inter}  gives a primary decomposition of  $J_S$. The ideals $I_S$ and $J_S$ are integrally closed  because the ideals appearing in their primary decomposition are symbolic powers of prime ideals and therefore integrally closed.  

The special case of Theorem \ref{primary} in which all $a_i$ are equal  has been proved in   \cite[Corollary 2.3]{BBC} and \cite[Theorem 3.3]{BBC}. It will be used in the proof of the theorem.  (Note however that in \cite{BBC} our ideal $I_S$ is denoted by $J_S$.) 

\begin{proof}[Proof of Theorem \ref{primary}]
By the definition of $e_{ub}(S)$ we have
$$
I_S \subset \bigcap_{u,b}I_u(b)^{e_{ub}(S)}.
$$
This implies the chain of inclusions
$$
\prod_{i=1}^w J_{t_i}(a_i) \subset J_S\subset \ini\biggl(\bigcap_{u,b}I_u(b)^{e_{ub}(S)}\biggr) \subset \bigcap_{u,b} \ini\bigl( I_u(b)^{e_{ub}(S)}\bigr)
= \bigcap_{u,b}J_u(b)^{e_{ub}(S)}
$$
where we have used Theorem \ref{classic} for the equality of the two rightmost terms.
If
\begin{equation}
\bigcap_{u,b}J_u(b)^{e_{ub}(S)}\subset \prod_{i=1}^w J_{t_i}(a_i))\label{crucial}
\end{equation}
as well, then we have equality throughout, implying \eqref{ini} and \eqref{ini-inter}. Then \eqref{inter} follows since two ideals with the same initial ideal must coincide if one is contained in the other. Therefore \eqref{crucial} is the crucial inclusion.

We prove it by induction on $w$. Let $M$ be a monomial in $\bigcap_{u,b}J_u(b)^{e_{ub}(S)}$. Then $M$ is contained in $J_{t_w}(a_w)$. This ideal is generated by all diagonals $\la f_1 \dots f_{t_w}\ra$ with $f_1\ge a_w$ by Theorem \ref{classic}(2). Among all these diagonals we choose the \emph{lexicographically smallest} and  call it $F$.  

Set $T=\bigl((t_1,a_1),\dots,(t_{w-1},a_{w-1})\bigr)$. It is enough to show that $M/F\in \bigcap_{u,b}J_u(b)^{e_{ub}(T)}$, and for this containment we must show $M/F\in J_u(b)^{e_{ub}(T)}$ for all $u$ and $b$. Evidently
$$
e_{ub}(T)=\begin{cases}
e_{ub}(S),& b \le a_w,\ u >t_w,\\
e_{ub}(S)-1,& b \le a_w,\ u \le t_w,\\
0,&  \text{else.}
\end{cases}
$$

If $e_{t_wb}(T)=0$, there is nothing to show. If $b \le a_w$, $u >t_w$, we have $e_{t_wb}(S)\ge e_{ub}(S)+1$ because $I_{t_w}(a_w)$ contributes to $e_{t_wb}(S)$, but not to $e_{ub}(S)$. This observation is important for the application of Lemma \ref{easy} that covers this case. The case $b \le a_w$, $u \le t_w$ is Lemma \ref{harder}.
\end{proof}

\begin{lemma}\label{easy}
Let $k\in \NN$.  Let $b\le a$ and $u>t$.  Let  $M\in J_u(b)^k\cap J_t(b)^{k+1}\cap J_t(a)$ be a monomial and let $F$  be  the lexicographic smallest diagonal of length $t$ that divides $M$.  Then $M/F\in J_u(b)^k$. 
\end{lemma}

\begin{proof}
We can apply \cite[Theorem 3.3]{BBC} to $J_u(b)^k\cap J_t(b)^{k+1}$: $M$ is divided by a product $D_1\cdots D_k E$ where $D_1,\dots,D_k$ are diagonals of length $u$ starting in column $b$ or further right, and $E$ is such a diagonal of length $t$. Even more: $J_u(b)^k\cap J_t(b)^{k+1}$ is generated by the initial monomials of the standard tableaux in $I_u(b)\cap I_t(b)^{k+1}$ (Proposition \ref{uni-init}). Therefore we can assume that $D_1\le_\str \dots\le_\str D_k\le_\str E$.

The greatest common divisor of $F$ and $D_1\cdots D_k E$ must divide $E$---if not we could replace $F$ by $F\vee E$ and obtain a lexicographically smaller diagonal of length $t$. Therefore $D_1\cdots D_k$ divides $M/F$.
\end{proof}

\begin{lemma}\label{harder}
Let $k\in \NN$.  Let $b\le a$ and $u\le t$.  Let  $M\in J_u(b)^k\cap J_t(a)$ be a monomial and let $F$  be  the lexicographic smallest diagonal of length $t$ that divides $M$.  Then $M/F\in J_u(b)^{k-1}$.
\end{lemma}

\begin{proof}
By Theorem \ref{classic} here exist diagonals $D_1,\dots,D_k$ such of length $u$ such that $D_1\cdots D_k$ divides $M$ and $D_1\le_\str\dots \le_\str D_k$. Division by $F$ can ``destroy'' more than one of these diagonals but, as we will see, the fragments can be joined to form $k-1$ diagonals of length $u$ as desired.

We explain the argument first by an example: $M=x_{11}x_{12}x_{23}x_{24}x_{34}\in J_2(1)^2\cap J_3(2)$. The lexicographically smallest diagonal of length $3$ is $\la 2 3 4 \ra$. It intersects both $2$-diagonals $\la 1 3\ra$ and $\la 2 4\ra$, but we can produce the new $2$-diagonal $\la 1 4\ra$ from the two fragments, and are done in this case: $M\in J_2(1)J_3(2)$. 

Let $r_1\le \dots\le r_p$ be the rows in which $F$ intersects one of the $D_i$, and choose $g_i$ maximal such that $F$ intersects $D_{g_i}$ in row $r_i$. In view of the order of the $D_i$ and by the choice of $F$ as the lexicographically smallest $t$-diagonal dividing $M$, we must have $g_{i+1}\le g_{i}$ for $i=1,\dots,p-1$. 

Every time that $F$ ``jumps'' to another diagonal, i.e., if $g_i>g_{i+1}$, we concatenate the entries in rows $1,\dots,r_i$ of $D_{g_{i+1}}$ with the entries in rows $r_i+1,\dots,u$ of $D_{g_{i}}$, thus producing a new diagonal. (Note that $F$ cannot return to $D_{g_{i+1}}$ in rows $\le r_i$ and has not touched $D_{g_{i}}$ in the other rows.) Only one $u$-diagonal is lost this way.
\end{proof}

Our next goal is to identify the irredundant components in the primary decomposition of $I_S$ described in Theorem \ref{primary}. To this end we prove the following facts. 

\begin{lemma}
\label{colon}
 Let $S$ be a NE-pattern and let $D$ be a subsequence of $S$.  Set $T=S\setminus D$.  Then $I_S:I_D=I_T$. 
\end{lemma}
\begin{proof} By induction on the cardinality of $D$, we may assume right away $D$ is a singleton.  Using Theorem \ref{primary} the desired equality boils down to the proof that for every $k>0$ one has $I_u(b)^k:I_t(a)=I_u(b)^{k-1}$ if $(b,u)\leq (t,a)$, and $I_u(b)^k:I_t(a)=I_u(b)^{k}$ otherwise. Both equalities follow from the fact that $I_u(b)$ has primary powers. 
\end{proof}

\begin{lemma}
\label{assopri}
Let $(t,a),(u,b)\in \NN_+^2$ such that   $t\le u$, $a \le b$ and $u+b\leq n+1$.  Then $I_t(a)I_u(b)\subset I_t(b)I_u(a)$. Actually, $I_t(a)$ is an associated prime to $R/I_t(b)I_u(a)$. 
\end{lemma}
\begin{proof}  For the inclusion $I_t(a)I_u(b)\subset I_t(b)I_u(a)$, in view of Theorem \ref{primary}  it is enough to show that $e_{vc}(S) \geq e_{vc}(T)$ for every $(v,c)$ where $S=\{(t,a), (u,b)\}$ and $T=\{(t,b), (u,a)\}$, and this is easy.  

The inclusion just proved shows one inclusion of the equality $\bigl(I_t(b)I_u(a)\bigr):I_u(b)=I_t(a)$. The other follows from the fact that $I_t(a)$ is prime and contains $I_t(b)I_u(a)$.  
\end{proof} 

Now we are ready to prove: 

\begin{proposition}
\label{irred} 
Given $S$, let  $Y$ be the set of the elements $(t,a)\in \NN_{+}^2$, $(t,a)\notin S$, such that there exists $(u,b)\in  \NN_{+}^2$ for which  $(t,b),(u,a)\in S$ and $t<u$, $a<b$. 
Then we have the following  primary decomposition:  
$$
I_S = \bigcap_{(v,c)\in S \cup Y}I_v(c)^{e_{vc}(S)}
$$ 
which is irredundant  provided all the points $(u,b)$ above can be taken so that $u+b\leq  n+1$. 
In particular, for fixed $S$, the given primary decomposition above is irredundant if $n$ is sufficiently large, and in this case all powers $I_S^k$ have the same associated prime ideals as $I_S$.
\end{proposition}

\begin{proof} The equality holds because of Theorem \ref{primary} and because if $(v,c)\not\in S\cup Y$ then $e_{vc}(S)$ is either equal to $e_{v+1,c}(S)$ 
or $e_{v,c+1}(S)$. 

It remains to show that the decomposition is irredundant  under the extra assumption. 
We can equivalently prove  that every prime $I_t(a)$ with $(t,a)\in S\cup Y$ is associated to $R/I_S$. For  $(t,a)\in S$ this follows from a general fact: for prime ideals $P_1,\dots,P_r\neq 0$ in a noetherian domain $A$ each $P_i$ is associated to $A/I$, $I=P_1\cdots P_r$. This follows easily by localization; one only needs that $IA_{P_i}\neq 0$ for all $i$.

Now let $(t,a)\in Y$ and  $(t,b),(u,a)\in S$ and such that $t<u$ and $a<b$ and $u+b\leq n+1$. Set $D=S\setminus \{ (t,b),(u,a)\}$. 
Then by  \ref{colon}  we have $I_S:I_D=I_t(b)I_u(a)$, and by \ref{assopri} $I_t(a)$ is  associated   to $R/I_t(b)I_u(a)$. It follows that  $I_t(a)$ is  associated   to $R/I_S$ as well.

For the last statement we note that the set $Y$ does not change if we pass from $I_S$ to $I_S^k$. 
\end{proof} 

Let us illustrate Theorem \ref{primary} and Proposition \ref{irred} by two examples. 

\begin{example} 
Let $n\geq 5$ and $S=\{(3,1),(3,3),(2,3),(1,4)\}$ so that 
$$
I_S=I_3(1)I_3(3)I_2(3)I_1(4).
$$

The ideal and the values $e_{ub}(S)$ are given by the following tables:
$$
	\begin{tabular}{|c|c|c|c|c|c|}
\hline
   & \phantom{$\bullet$ }  & & $\bullet$ &  \phantom{$\bullet$} &  \phantom{$\bullet$} \\ \hline
 & & $\bullet$ & & &  \phantom{$\bullet$} \\ \hline
 $\bullet$ & & $\bullet$ & & &  \phantom{$\bullet$} \\ \hline
\end{tabular}	
\qquad\qquad 
 \begin{tabular}{|c|c|c|c|c|c|}
\hline
\textcircled{4} & 3 & \textcircled{3} & \boxed{1} &  0 & 0\\ \hline
\textcircled{3} & 2&   \boxed{2} & 0 & 0 & 0 \\ \hline
\boxed{2} & 1&  \boxed{1} & 0 & 0 & 0 \\ \hline
\end{tabular}	
$$
The  boxed and circled  values are the essential ones and  give rise to the irredundant components.  
The boxed  correspond to elements in $S$ and the circled   to elements in $Y$.  
Hence a irredundant primary decomposition of $I_S$ is: 
$$
I_S=I_1(4) \cap  I_1(3)^3  \cap I_2(3)^2  \cap I_3(3)  \cap I_1(1)^4  \cap I_2(1)^3  \cap I_3(1)^2.
$$
\end{example}

\begin{example} 
If the criterion in Proposition \ref{irred} does not apply, $I_t(a)$ may nevertheless be associated to $R/I_S$. We choose $n=4$.
	
First, let $S=\{(3,1),(2,2),(1,3)\}$.  Then $(1,1)\in Y$ and the corresponding ``$(u,b)$" is $(3,3)$ which does not satisfy $u+b\leq n+1$. Unexpectedly, $I_1(1)$ is associated to $I_S$. 

Second, let   $S=\{(3,1),(1,3)\}$.  Again $(1,1)\in Y$, and it has the same corresponding $(u,b)$. 
But in this case $I_1(1)$ is  not  associated to $I_S$. 
\end{example}

\section{The northeast straightening law} 

It is now crucial to have a ``normal form'' for elements of $I_S$. For this purpose we select a $K$-basis that involves  the natural system of generators, the NE-tableaux $\Delta$ of pattern $S$. It has already become apparent in the proof of Lemma \ref{harder} that we cannot simply require that $\Delta$ is a standard tableau, and the following example for $S=\bigl( (1,2),(3,3)\bigr)$shows this explicitly:
$$
[1 4][2 3 4]=[1 3 4][2 4]-[1 2 4][3 4].
$$
The difficulty is that the transformations occurring in the standard straightening procedure do not respect the bounds of the NE-ideals in general. However, they do so in an important special case to which we will come back.

Let $M\Delta$ be a NE-tableau. A monomial $\la c_1,\dots,c_t\ra$ is called a \emph{diagonal of type (t,a) in $M\Delta$} if $c_1\ge a$ and $\la c_1,\dots,c_t\ra \mid \ini(M\Delta)=M\ini(\Delta)$.

\begin{definition}
Let $S=\bigl((t_1,a_1),\dots,(t_w,a_w)\bigr)$ be a NE-pattern. A NE-tableau $M\Delta$ of pattern $S$, $\Delta=\delta_1\cdots\delta_w$, is \emph{$S$-canonical} if
$\ini(\delta_j)$ is the lexicographically smallest diagonal of type $(t_j,a_j)$ in the NE-tableau $M\delta_1\cdots\delta_j$ of pattern $\bigl((t_1,a_1,\dots,(t_j,a_j)\bigr)$ for $j=1,\dots,w$. 
\end{definition}

As an example we consider the monomial $M=x_{11}x_{12}x_{13}x_{23}x_{24}x_{25}x_{35}$, graphically symbolized by the following table:
\begin{center}
\begin{tabular}{|c|c|c|c|c|}
\hline
$\bullet$ & $\bullet$ & $\bullet$ & &\\ \hline
& & $\bullet$ & $\bullet$ & $\bullet$\\ \hline
&&&&$\bullet$ \\
\hline
\end{tabular}	
\end{center}
It depends on the pattern $S$ which $S$-canonical tableau has $M$ as its initial monomial.
\begin{enumerate}
\item For $S=\bigl((2,1),(3,2),(2,2) \bigr)$ the canonical tableau with  initial monomial $M$ is
$$
[1 3][2 4 5][3 5].
$$
\item For $S=\bigl((2,1),(2,2),(3,3) \bigr)$ it is
$$
[1 3][2 5][3 4 5].
$$
\item For $S= \bigl((2,1),(2,2),(2,3) \bigr)$ it is
$$
x_{35}[1 3][2 4][3 5].
$$
\end{enumerate}

Note that different canonical NE-tableaux of the same pattern are $K$-linearly independent since they have different initial monomials: if the pattern $S$ is fixed, then an $S$-canonical tableau is uniquely determined by its initial monomial. In fact, the diagonals that are split off successively are uniquely determined, and each diagonal belongs to a unique minor.

We can now formulate the \emph{NE-straightening law}:

\begin{theorem}\label{NEstr}
Let $S=\bigl((t_1,a_1),\dots,(t_w,a_w)\bigr)$ be a NE-pattern and $x\in I_S$. then there exist uniquely determined $S$-canonical NE-tableaux $M_i\Gamma_i$, $i=0,\dots,p$,  and coefficients $\lambda_i\in K$ such that
$$
x=\lambda_0M_0\Gamma_0+\lambda_1M_1\Gamma_1+\dots+\lambda_pM_p\Gamma_p
$$
and
$$
\ini(x)=\ini(M_0\Gamma_0)>\ini(M_1\Gamma_1)>\dots> \ini(M_p\Gamma_p).
$$
\end{theorem}

\begin{proof}
Let $\lambda_0$ be the initial coefficient of $x$. It is enough to show the existence of $M_0\Gamma_0$ since $\ini(x-\lambda_0M_0\Gamma_0)<\ini(x)$, and we can apply induction.	

Clearly $\ini(x)\in\ini(I_S)$. The factorization of the monomial $\ini(x)$ constructed recursively in the proof of Theorem \ref{primary} is exactly the factorization that gives it the structure $M_0\ini(\Gamma_0)$ for an $S$-canonical NE-tableau of pattern $S$: it starts by extracting the lexicographically smallest diagonal $D_w$ of length $t_w$ from $\ini(x)$, and applies the same procedure to $\ini(x)/D_w$ recursively. As pointed out above, this factorization belongs to a unique $S$-canonical tableau.
\end{proof}

We call Theorem \ref{NEstr} a straightening law, since it generalizes the ``ordinary'' straightening law to some extent:

\begin{theorem}\label{nested}
Suppose the NE-pattern $S=\bigl((t_1,a_1),\dots,(t_w,a_w)\bigr)$ satisfies the condition\discuss{CH} $t_i\ge t_{i+1}$ for $i=1,\dots,w-1$, and let $\Delta$ be pure NE-tableau of pattern $S$. Then the representation
$$
\Delta=\Delta_0+\lambda_1\Sigma_1+\dots+\lambda_p\Sigma_p
$$
of Theorem \ref{straight}(2) is the $S$-canonical representation.
\end{theorem}

\begin{proof}
The only question that could arise is whether the representation is $S$-canonical. It is successively obtained from $\Delta$ by applying the straightening law to airs of minors:
$$
\delta_\sigma=(\delta\wedge \sigma)(\delta_\vee\sigma)+\lambda_1\eta_1\zeta_1+\dots+\lambda_q\eta_q\zeta_q,\qquad
\eta_i\le_\str \delta_\wedge\sigma, \ \delta_i\vee\sigma \le_\str \zeta_i, \ i=1,\dots,q,
$$
as in Theorem \ref{straight}(1). Therefore it is enough to consider the case $w=2$, $S=((t,a),(u,b))$, $a\le b$, $t\ge u$, $\delta=[d_1\dots d_t]$, $\sigma=[s_1 \dots s_u]$. The smallest column index is $\min(d_1,s_1)\ge a$. So all minors $\zeta_i$ belong to $I_t(a)$. The minors $\eta_i$ satisfy the inequalities $\eta_i\ge_\str \delta$ and $\eta_i\ge_\str \sigma$. Therefore they belong to $I_u(b)$.
\end{proof}

For later use we single out two special cases of Theorem \ref{NEstr}. 
\begin{enumerate}
\item For $\delta\in I_t(a)$, $|\delta|=t$, $\sigma\in I_u(b)$, $|\sigma|=u$, there is an equation
\begin{equation}
\delta_\sigma=\delta_0\sigma_0+\lambda_1\delta_1\sigma_1+\dots \lambda_p\delta_p\sigma_p\label{deg2_1}
\end{equation}
with $\lambda_1,\dots,\lambda_p\in K$ and canonical NE-tableaux $\delta_0\sigma_0,\dots\delta_p\sigma_p$ of pattern $((t.a),(u,b))$ and $\ini(\delta\sigma)=\ini(\delta_0\sigma_0)>\ini(\delta_1\sigma_1)+\dots\ini(\delta_p\sigma_p)$.

\item With the same notation for $\delta$, for every indeterminate $x_{uv}$ there is an equation
\begin{equation}
x_{uv}\delta=x_{u_0v_0}\delta_0+\lambda_1x_{u_1v_1}\delta_1+ \dots+x_{u_pv_p}\delta_p\label{deg2_2}
\end{equation} 
with $\lambda_1,\dots,\lambda_p\in K$, $x_{u_0v_0}\delta_0,\dots,x_{u_pv_p}\delta_p$ canonical of pattern $(t,a)$ and $\ini(x_{uv}\delta)=\ini(x_{u_0v_0}\delta_0)> \ini(x_{u_1v_1}\delta_1)>\ini(x_{u_pv_p}\delta_p)$. 
\end{enumerate}

Equation \eqref{deg2_2} is nothing but a linear syzygy of $t$-minors (unless it is a tautology). These syzygies have sneaked in through the use of the theorem that the $t$-minors form a  Gr\"obner  basis of $I_t(a)$.

We complement the discussion of canonical decompositions by showing that a non-canonical tableau can be recognized by comparing the factors pairwise.

\begin{lemma}\label{lem4}
Let $S$ be a NE-pattern and $M\Delta$, $\Delta=\delta_1\cdots \delta_w$, be a NE-tableau of pattern $S$. If $M\Delta$ is not $S$-canonical, then at least one of the following two cases occurs:
\begin{enumerate}
\item there exist a visor $x_{ij}$ of $M$ and an index $k$ such that $x_{ij}\delta_k$ is not NE-canonical of pattern $(t_k,a_k)$;
\item there exist indices $q$ and $k$, $q<k$, such that $\delta_q\delta_k$ is not $((t_q,a_q),(t_k,a_k))$-canonical.
\end{enumerate}
\end{lemma}

\begin{proof}
We choose $k$ maximal with the property that $\ini(\delta_k)$ is not the lexicographically smallest $(t_k,a_k)$-diagonal dividing $\ini(M\delta_1\dots\delta_k)$. Set $t=t_k$, $\delta=[d_1 \dots d_t]$ and let $\la e_1 \dots e_t\ra$ be the lexicographically smallest such diagonal. Then choose $r$ maximal with the property that $d_r<e_r$. Since $x_{re_r}$ divides $\ini(M\delta_1\cdots \delta_k)$, at least one of the following two cases must hold:
\begin{enumerate}
\item $x_{re_r}\mid M$;
\item $x_{re_r}\mid \ini(\delta_q)$ for some $q<k$.
\end{enumerate} 
In the first case $x_{re_r}\delta_k$ is not $(t_k,a_k)$-canonical, and in the second case $\delta_q\delta_k$ fits case (2) of the lemma: $\la d_1 \dots d_{r-1} e_r d_{r+1} \dots e_t\ra $ is lexicographically smaller than $\la d_1 \dots d_t\ra$ and divides $\ini(x_{re_r}\delta_k)$ or $\ini(\delta_q\delta_k)$, respectively.
\end{proof}

Lemma \ref{lem4} and the equations \eqref{deg2_1} and \eqref{deg2_2} indicate that the $S$-canonical representation of an element $x\in I_S$ can be obtained by the successive application of quadratic relations. This is indeed true and will be formalized in the next section.

\section{The multi-Rees algebra}
The natural object for the simultaneous study of the ideals $I_S$ is the multi-Rees algebra
$$
\Rees=R[I_t(a)T_{ta}: 1\le t, a\le n,\    t+a\le n+1]
$$
where the $T_{ta}$ are new indeterminates It is a subalgebra of the polynomial ring 
$$R[T_{ta}: 1\le t, a\le n,\    t+a\le n+1],$$
 and the products of the ideals $I_t(a)$ appear as the coefficient ideals of the monomials in the indeterminates $T_{ta}$. These monomials are parametrized by the patterns $S$: for $S=((t_1,a_1),\dots,(t_w,a_w))$ we set
$$
T^S=T_{t_1a_1}\cdots T_{t_wa_w}.
$$
Then
$$
\Rees=\Dirsum_{S} I_ST^S.
$$
The monomial order on $R$ is extended to $R[T_{ta}: 1\le a\le n,\ 1\le t+a\le n+1]$ in an arbitrary way. The extension will be denoted by $<_\lex$ as well.
 
Alongside with $\Rees$ we consider the multi-Rees algebra $\Rees_{\ini}$ defined by the initial ideals $J_t(a)$:
$$
\Rees_{\ini}=R[J_t(a)T_{ta}: 1\le a\le n,\ 1\le t+a\le n+1].
$$
As always in this context, there is a second `` initial'' object that comes into play, namely the initial subalgebra of $\Rees$:
$$
\ini(\Rees)=\Dirsum_{S} J_ST^S.	
$$
(Recall that $J_S=\ini(I_S)$ by definition.) From Theorem \ref{primary} one can easily derive a first structural result on $\Rees$ and $\Rees_{\ini}$.

\begin{theorem}\label{Rees-1}
\leavevmode
\begin{enumerate}
\item With respect to any extension of the monomial order on $R$ to $R[T_{ta}: 1\le a\le n,\ 1\le t+a\le n+1]$ one has $\ini(\Rees)=\Rees_{\ini}$.
\item $\Rees$ and $\Rees_{\ini}$ are normal Cohen-Macaulay domains.
\end{enumerate}
\end{theorem}

\begin{proof}
The equation $\ini(\Rees)=\Rees_{\ini}$ is just equation \eqref{ini} read simultaneously for all NE-patterns $S$.

The normality of $\Rees$ and $\Rees_{\ini}$ follows from the fact that all the ideals $I_S$ and $J_S$ are integrally closed by Theorem \ref{primary}. We observe that $\Rees_{\ini}$ is a normal monoid domain, and therefore Cohen-Macaulay by Hochster's theorem.  Finally   we use the transfer of  the Cohen-Macaulay property from $\ini(\Rees)=\Rees_{\ini}$ to $\Rees$, see \cite{CHV}. 	
\end{proof}
 
In order to gain insight into the minimal free resolutions of the ideals $I_S$ over $R$ we must understand $\Rees$ as the residue class ring of a polynomial ring over $K$. To this end we introduce a variable $z_{a\delta}$ for every bound $a$ and every $t$-minor $\delta\in I_t(a)$. Let
$$
\cS=R[z_{a\delta}: |\delta|=t,\ t+a\le n+1,\ \delta \in I_t(a)].
$$
Viewed as a $K$-algebra, $\cS$ needs also the indeterminates $x_{ij}$, $1\le i\le m$, $1\le j\le n$.  We want to study the surjective $R$-algebra homomorphism
$$
\Phi:\cS\to\Rees, \qquad \Phi(z_{a\delta})=\delta T_{a|\delta|},\quad \Phi|R=\operatorname{id}.
$$
We introduce an auxiliary  monomial order  on $\cS$ by first ordering the indeterminates. The $x_{ij}$ are ordered as in $R$. Next we set
$$
x_{ij}>z_{a\delta}
$$
for all $i,j,a,\delta$, and 
$$
z_{a\delta} > z_{b\sigma} \quad\iff\qquad a <b \quad\text{or } a=b \text { and } \ini(\delta)>_\lex \ini(\sigma) .
$$
This order of the indeterminates is extended to the \emph{reverse} lexicographic order $\le_\revlex$ it induces on the monomials in $\cS$.

Now we define the main monomial order on $\cS$ as follows. For monomials $Z_1$ and $Z_2$ in $x_{ij}$ and $z_{a\delta}$ we set
\begin{multline*}
Z_1 \prec Z_2\quad\iff\quad \ini(\Phi(Z_1))<_\lex \ini(\Phi(Z_2)) \quad\text{or }\\ \ini(\Phi(Z_1))= \ini(\Phi(Z_2)) \text { and } Z_1<_\revlex Z_2.\end{multline*}
In other words, we pull the monomial order on $\Rees$ back to $\cS$ and then use our auxiliary order to separate monomials with the same image under $\Phi$.

\begin{theorem}\label{Rees-2}
\leavevmode
\begin{enumerate}
\item With respect to the monomial order $\prec$ the ideal $\cI=\Ker\Phi$ has a  Gr\"obner  basis of quadrics, given by the equations \eqref{deg2_1} and \eqref{deg2_2} (interpreted as elements in $\cI$).
\item In particular $\Rees$ is a Koszul algebra.
\item All the ideals $I_S$ have linear minimal free resolutions over $R$.
\end{enumerate}

\end{theorem}

\begin{proof}
Equation \eqref{deg2_1} defines the polynomial
$$
z_{a\delta}z_{b\sigma}-\bigl( z_{a\delta_0}z_{b\sigma_0}+\lambda_1z_{a\delta_1}z_{b\sigma_1} +\dots+\lambda_wz_{a\delta_w}z_{b\sigma_w} \bigr)
$$
in $\cI$. By the definition of $\prec$ we first observe that only $z_{a\delta}z_{b\sigma}$ or $z_{a\delta_0}z_{b\sigma_0}$ can be the initial monomial. But then the auxiliary reverse lexicographic oder makes $z_{a\delta}z_{b\sigma}$ the leading monomial. Similarly one sees that $x_{uv}z_{a\delta}$ is the leading monomial of the polynomial in $\cI$ defined by \eqref{deg2_2}. 

Let $\cJ$ be the ideal generated by all monomials $Z=Mz_{a_1\delta_1}\cdots z_{a_w\delta_w}$, $M\in R$, for which $M\delta_1\cdots\delta_w$ is not canonical of pattern $S=((|\delta_1|,a_1),\dots,(|\delta_1|,a_1))$. It follows from Lemma \ref{lem4} that $MMz_{a_1\delta_1}\cdots z_{a_w\delta_w}$ then contains a factor $z_{a_i\delta_i}z_{j\delta_j}$ or a factor $x_{uv}z_{a_j\delta_j}$ that is not mapped to a canonical tableau of the associated pattern. In connection wit the argument above, this observation implies that $\cJ\subset\ini_\prec(\cI)$.

On the other hand the set of monomials that do not belong to $\cJ$ form a $K$-basis of $\cS/\cI=\Rees$ by Theorem \ref{NEstr}. This is only possible if $\cJ=\ini(\cI)$.

Since $\cI$ has a  Gr\"obner  basis of quadratic polynomials, $\Rees=\cS/\cI$ is a Koszul algebra. By (the multigraded version of) a  theorem of Blum \cite{B} the linearity of the resolutions  follows from the Koszul property of the multi-Rees algebra.	
\end{proof}

In addition to $\Phi$, we have a surjective $R$-algebra homomorphism
$$
\Psi:\cS\to\Rees_{\ini}, \qquad \Psi(z_{a,\delta})=\ini(\delta) T_{a|\delta|},\quad \Phi|R=\operatorname{id}.
$$

\begin{theorem}\label{Rees-3}\leavevmode
\begin{enumerate}
\item With respect to the monomial order $\prec$ the ideal $\cB=\Ker\Psi$ has a  Gr\"obner  basis of quadrics, given by the binomials resulting from the equations $\ini(\delta\sigma)=\ini(\delta_0\sigma_0)$ in \eqref{deg2_1} and $\ini(x_{uv}\delta_)=\ini(x_{u_0v_0}\delta_0) in $\eqref{deg2_2} (interpreted as elements in $\cI$).
\item In particular $\Rees_{\ini}$ is a Koszul algebra.
\item All the ideals $J_S$ have linear minimal free resolutions over $R$.
\end{enumerate}
\end{theorem}

\begin{proof}
The first statement is proved completely analogous the first statement in Theorem \ref{Rees-2}, and the second and third follow from it in the same way as for Theorem \ref{Rees-2}.
\end{proof}

\section{Some Gorenstein Rees rings and some factorial fiber rings}

In this section we will consider multi-Rees algebras defined by some of the ideals $I_t(a)$. More generally, if $I_1,\dots,I_p$ are ideals in $R$, we let
$$
R(I_1,\dots,I_p)=R[I_iT_i:i=1,\dots,p]\subset R[T_1,\dots,T_p]
$$
denote the multi-Rees algebra defined by $I_1,\dots,I_t$. Note that we could as well have defined it by taking ordinary Rees algebras successively, since
$$
R(I_1,\dots,I_p)=B\bigl(I_pB\bigr) \quad \mbox{ where } \quad  B=R(I_1,\dots,I_{p-1}).
$$ 

Some of the Rees rings defined by NE-ideals of minors are Gorenstein. This is not true in general for the ``total'' multi-Rees rings of the last section: the first potential non-Gorenstein example is a $2\times 3$-matrix, and the corresponding total multi-Rees ring is indeed not Gorenstein. Nevertheless, the multi-Rees rings defined by certain selections of the ideals $I_t(a)$ are Gorenstein, as we will see in the following. 

The ideals $I_t(a)$ form a poset under inclusion. The minimal elements are the principal ideals $I_t(n-t+1)$ and the maximal element is $I_1(1)$, the ideal generated by the indeterminates in the first row of our matrix $X$. The next theorem states the Gorenstein property of the multi-Rees algebras defined by an unrefinable ascending chain in our poset that starts from a minimal element or a cover of a minimal element.

\begin{theorem}\label{Rees-Gor} Let $I_{t_1}(a_1)\subset I_{t_2}(a_2) \subset \dots \subset  I_{t_p}(a_p)$ such that $\height I_{t_1}(a_1)=1$ or $2$ and  $\height I_{t_i}(a_i)=1+\height I_{t_{i-1}}(a_{i-1})$ for $i=2,\dots,p$. Equivalently, 
\begin{enumerate}
\item $a_1=n-t_1$ or $a_1=n-t+1$;
\item $t_i=t_{i-1}$ and $a_i=a_{i-1}-1$ or $t_i=t_{i-1}-1$ and $a_i=a_{i-1}$ for $i=2,\dots,p$. 
\end{enumerate}
Then the multi-Rees algebra $R\bigl(I_{t_1}(a_1),\dots,I_{t_p}(a_p)\bigr)$ is Gorenstein and normal  with divisor class group $\ZZ^{p-1}$ or $\ZZ^p$, depending on whether $a_1=n-t_1$ or $a_1=n-t+1$. 
\end{theorem} 

\begin{proof}
Note that the smallest ideal is a principal ideal if $a_1=n-t_1+1$. In this case $R\bigl(I_{t_1}(a_1),\dots,I_{t_p}(a_{t_p})\bigr)$ is just (isomorphic to) a polynomial ring over $R\bigl(I_{t_2}(a_2),\dots,I_{t_p}(a_p)\bigr)$, and $a_2=n-t_2$. Since polynomial extensions do not affect the Gorenstein property, we can assume that $a_1=n-t_1$.

Let $\Rees=R\bigl(I_{t_1}(a_1),\dots,\allowbreak I_{t_p}(a_p)\bigr)$, $\Rees'=R\bigl(I_{t_1}(a_1),\dots,I_{t_{p-1}}(a_{p-1})\bigr)$, and $Q=I_{t_p}(a_p)\Rees'$. Then $\Rees$ is just the ordinary Rees algebra of the ideal $Q$ of $\Rees'$, and by induction on $p$ it is enough to understand the extension of $\Rees'$ to $\Rees$.

By the next lemma, $Q$ is a prime ideal of height $2$ such that $Q\Rees'_Q$ is generated by $2$ elements. Moreover, $\Rees'$ and $\Rees$ are normal domains since they are retracts of the total multi-Rees algebra of the last section (or by Theorem \ref{primary}). Under these conditions a theorem of Herzog and Vasconcelos \cite[Theorem(c), p. 183]{HV} shows that the canonical module of $\Rees$ has the same divisor class as the canonical module of $\Rees'$ (extended to $\Rees$):
$$
\cl(\omega_{\Rees})=\cl(\omega_{\Rees'})+(\hht Q-2)\cl(Q\Rees)=\cl(\omega_{\Rees'})\in \Cl(\Rees)=\Cl(\Rees')\dirsum \ZZ.
$$
By induction $\Rees'$ is Gorenstein, $\cl(\omega_{\Rees'})=0$ . Therefore $\Rees$ is Gorenstein as well, and we are done.
The assertion on the divisor class group follows as well. \end{proof} 

\begin{lemma}\label{prime}
With the notation of the preceding proof, $Q$ is a prime ideal of height $2$ in $\Rees'$ such that $Q\Rees'_Q$ is generated by $2$ elements.  
\end{lemma}

\begin{proof}
The most difficult claim is the primeness of $Q$. We show primeness of a larger class of ideals, namely all ideals $I_u(b)\Rees'$ such that $I_u(b)\supset I_{t_{p-1}}(a_{p-1})$. Set $P=I_u(b)$. 

As an auxiliary ring we consider the multi-Rees algebra $\cS=R(P,\dots,P)$ with $p-1$ ``factors'' $P$. For $\Rees'$ as above one has $\cS\supset \Rees'$ since $P$ contains all the ideals defining $\Rees'$. It follows from Equation (\ref{inter}) that 
$$P\Rees'=P\cS\cap \Rees'.$$ In fact, both algebras use the variables $T_1,\dots,T_{p-1}$. The coefficient ideal of $T_1^{e_1}\cdots T_{p-1}^{e_{p-1}}$ in $ P\cS$ is $P^{1+e_1+\dots+e_{p-1}}$ and its coefficient ideal  in $\Rees'$ is 
$$
I_{t_1}(a_1)^{e_1}\cdots I_{t_{p-1}}(a_{p-1})^{e_{p-1}}
$$
whereas the coefficient ideal in $P\Rees'$ is $PI_{t_1}(a_1)^{e_1}\cdots I_{t_{p-1}}(a_{p-1})^{e_{p-1}}$. Equation \eqref{inter} implies
$$
PI_{t_1}(a_1)^{e_1}\cdots I_{t_{p-1}}(a_{p-1})^{e_{p-1}}=P^{1+e_1+\dots+e_{p-1}}\cap I_{t_1}(a_1)^{e_1}\cdots I_{t_{p-1}}(a_{p-1})^{e_{p-1}},
$$
and this is the desired equality.

The primeness of $P\Rees'$ follows if $P\cS$ is a prime ideal. The algebra $\cS$ is the Segre product of the polynomial ring in $p-1$ variables over $K$ and the ordinary Rees algebra $S=R[PT]$. Consequently $\Rees'/P\Rees'$ is the Segre product of the same polynomial ring and the associated graded ring $S/PS$ of $P$. But the latter is an integral domain \cite[(9.17)]{BV}.

The smallest  choice for $P$ is $I_{t_{p-1}}(a_{p-1})$. This nonzero prime ideal is properly contained in $Q$. Therefore $\hht Q\ge 2$. In order to finish the proof it remains to show that $Q\Rees'_Q$ is generated by $2$ elements. The indeterminate $x_{1n}$ in the right upper corner of $X$ is not contained in $Q$ if $t_p>1$. We can invert it and, roughly speaking, reduce all minor sizes and $n$  by $1$. This is a standard localization argument; see \cite[(2.4)]{BV} (where it is given for $x_{11}$). Therefore we can assume that $t_p=1$.

If even $p=1$, then $a_1=n-1$, and $P$ is evidently generated by $2$ elements. So suppose that $p>1$. There are two cases left, namely $t_{p-1}=1$, $a_{p-1}=a_p-1$ or $t_{p-1}=2$, $a_{p-1}=a_p$.

In the first case we use the equations
$$
x_{1i}(x_{1n}T_{p-1})=x_{1n}(x_{1i}T_{p-1}), \qquad i\ge a_{p-1}=a_p+1.
$$
The element $x_{1n}T_{p-1}$ does not belong to $Q$, and becomes a unit in $\Rees'_Q$. Thus $Q\Rees'_Q$ is generated $x_{1a_p}$ and $x_1n$.

In the other case one uses the linear syzygies of the $2$-minors in $I_2(a_{p-1})$ with coefficients from the first row of $X$ in order to show that $Q\Rees'_Q$ is generated by $x_{1n-1}$ and $x_{1n}$.
\end{proof}

\begin{remark}
(a) An alternative proof of Theorem \ref{Rees-Gor} can be given by toric methods. Using Theorem \ref{in-powwer-primary} one can describe the cone of the exponent vectors of $\ini(\Rees)$ ($\Rees$ as in Theorem \ref{Rees-Gor}) by inequalities. These inequalities have coefficients in $\{0,\pm 1\}$, and $1$ occurs exactly one more time than $-1$. Therefore the exponent vector with all entries $1$ generates the interior of the cone of exponent vectors, which is the set of exponent vectors of the canonical module of $\ini(\Rees)$ by theorem of Danilov and Stanley \cite[6.3.5]{BH}. It follows that $\ini(\Rees)$ is Gorenstein and therefore $\Rees$ is also Gorenstein. 

The opposite implication also works for the Gorenstein property since $\ini(\Rees)$ is known to be Cohen-Macaulay. For Cohen-Macaulay domains the Gorenstein property only depends on the Hilbert series by a theorem of Stanley \cite[4.4.6]{BH}.

(b) In general, extensions of the prime ideals $I_t(a)$ to Rees algebras defined by collections of the ideals $I_u(b)$ are not prime. However, by extending the intersection argument in the proof of Lemma \ref{prime} one can show that they are radical ideals.
\end{remark}

A Cohen-Macaulay factorial domain is Gorenstein. So one my wonder whether the Rees rings discussed above can be factorial. But, apart from trivial exceptions, Rees rings cannot be factorial.  On the other hand, the fiber rings have more chances to be factorial. 
 The fiber ring $F(I_1,\dots, I_p)$ of associated to the multi-Rees ring of  ideals $I_i$, $i=1,\dots,p$ is defined as 
$$
F(I_1,\dots, I_p)=R(I_1,\dots,I_p)/\mm R(I_1,\dots,I_p)
$$
where $\mm$ is the irrelevant maximal ideal of $R$. If each of the ideals $I_i$ is generated by elements of the same degree, say  $d_i$, the multi-fiber ring is a retract of the Rees ring, namely  
$$
F(I_1,\dots, I_p)=K[(I_i)_{d_i}T_i: i=1,\dots,p]
$$
where $(I_i)_{d_i}$ is the homogeneous component of degree $d_i$. It can of course be replaced by a system of degree $d_i$ generators of $I_i$.

Let us consider a sequence $(t_1,a_1),\dots,(t_p,a_p)$ such that $t_1<\cdots <t_p$ and $I_i=I_{t_i}(a_i)$. In this case  the multi-fiber ring can even be identified with the subalgebra
\begin{equation}
K[(I_i)_{t_i}: i=1,\dots,p]\label{fiber-emb}
\end{equation}
of $R$ (it is only essential that the degrees $t_i$ are pairwise different). Thus the multi-fiber ring is a subalgebra of the homogeneous coordinate ring of the flag variety of $K^n$. The latter is the subalgebra of $K[X]$ (where $X$ is an $n\times n$-matrix) generated by the $t$-minors of the first $t$ rows, $t=1,\dots,n$. The coordinate ring of the flag variety is factorial. See \cite[p.~138]{Ful} for an invariant-theoretic argument; an alternative proof is given below.

\begin{theorem}
Let $t_1<\cdots < t_p$ and $a_1\ge \dots\ge a_p$ and $I_i=I_{t_i}(a_i)$ for $i=1,\dots,p$. Then the multi-fiber ring $F(I_1,\dots, I_p)$  is factorial and therefore Gorenstein.
\end{theorem}

\begin{proof}
Set $F=F(I_1,\dots, I_p)$. In the first step we reduce the claim to the special case in which $p=n$ and $t_i=i$ for $i=1,\dots,n$. Starting from the given data, we augment $X$ to have at least $n$ rows. Changing the indeterminates for the embedding of $F$ into a polynomial ring over $R$, we can assume that $F=K[(I_{t_i}(a_i))_{t_i}T_{t_i}]$. Then we let $G$ be the multi-fiber ring defined by $(1,b_1),\dots,(n,b_n)$ where $b_i=a_1$ if $i<t_1$, $b_i=a_j$ if $t_j\le i < t_{j+1}$, $b_j=a_t$ for $j> t_p$. Consider the $R$-endomorphism $\Phi$ of $R[T_1,\dots,T_n]$ that maps all indeterminates $T_{t_i}$ to themselves and the other $T_j$ to $0$. Then $\Phi$ is the identity on $F$ and maps $G$ onto $F$. Thus $F$ is a retract of $G$. Since retracts of factorial rings are factorial, it is enough to consider $G$, and  we have reduced the general claim to the special case in which $p=n$ and $t_i=i$ for $i=1,\dots,n$. Moreover, we can use the embedding \eqref{fiber-emb} to simplify notation.

Using the NE straightening law  for pure (!) NE-tableaux one sees that $x_{1n}$ is a prime element in $F$. By the theorem of Gauß-Nagata, the passage to $F[x_{1n}^{-1}]$ does not affect factoriality. 

We repeat the localization argument of the proof of Theorem \ref{Rees-Gor}. Note that the linear syzygies of the $t$-minors in $I_t(a_t)$ with coefficients $x_{1i}$ are polynomial equations of the algebra generators of $F$ since $a_1\ge a_j$ for $j=1,\dots,n$. It follows that $F[x_{1n}^{-1}]$ is a multi-fiber ring defined by minors of sizes $1,\dots,n-1$ over a Laurent polynomial ring. This does not harm us since we can replace $K$ by a factorial ring of coefficients right from the start. This concludes the proof that $F$ is factorial. As we will remark in \ref{final}, $F$ is a Cohen-Macaulay domain, so we may conclude it is Gorenstein by virtue of Murthy's theorem \cite[3.3.19]{BH}. 
\end{proof}

Note that the theorem covers the flag variety coordinate ring for which all the bounds $a_i$ are equal to $1$.

\begin{remark}
\label{final}
(a) In general the multi-fiber ring  $F(I_{t_1}(a_1),\dots,I_{t_p}(a_p))$ is not factorial. For example, for $t+2\leq n$ factoriality fails for  $F(I_t(1), I_t(2))$  because  of the Segre-type relations $(fT_1)(gT_2)=(gT_1)(fT_2)$   for distinct $t$-minors $f,g$ in  $I_t(2)$. 

(b) The multi-fiber ring $F(I_{t_1}(a_1),\dots,I_{t_p}(a_p))$ is  a Cohen-Macaulay normal domain  for every $(t_1,a_1),\dots,(t_p,a_p)$, as can be seen via deformation to the initial algebras. 

(c)  In general $F(I_{t_1}(a_1),\dots,I_{t_p}(a_p))$  is not  Gorenstein, for example  $F(I_1(1),I_1(2))$  is not Gorenstein when $n\geq 4$.  On the other hand, there is strong experimental evidence that the multi-fiber rings defined by sequences $(t_1,a_1),\dots,(t_p,a_p)$ as in Theorem \ref{Rees-Gor} are Gorenstein.  In the case in which the $t_i$ are all equal, say equal to $t$,  this is clearly true because the initial algebra of $F(I_t(1),I_t(2),\dots, I_t(n+1-t))$ coincides with  the initial algebra of the coordinate ring  of the Grassmannian $G(t+1,n+1)$. 
\end{remark}

\end{document}